\documentclass[a4paper, reqno]{amsart}

\usepackage{amsthm}
\usepackage{amsfonts}
\usepackage{amsmath,amsthm,amssymb}
\usepackage{comment}
\usepackage{fancybox}
\usepackage{epsf}
\usepackage[abs]{overpic}
\usepackage{layout}
\usepackage{bm}

\newtheorem{theorem}{Theorem}[section]
\newtheorem{lemma}[theorem]{Lemma}
\newtheorem{proposition}[theorem]{Proposition}
\newtheorem{corollary}[theorem]{Corollary}

\theoremstyle{definition}
\newtheorem{definition}[theorem]{Definition}
\theoremstyle{remark}
\newtheorem{remark}[theorem]{Remark}
\newtheorem{example}[theorem]{Example}
\newtheorem{conjecture}{Conjecture}[section]

\numberwithin{equation}{section}

\graphicspath{{figures/}}

\begin{document} 

\title[Linking number of grid models]{Linking number of grid models}

\author[Senja Barthel]{Senja Barthel} 
\address[Barthel]{
Department of Mathematics, Vrije Universiteit Amsterdam, Center for Topology and Applications Amsterdam (CTA$^2$), De Boelelaan 1111,
1081 HV Amsterdam, Netherlands
}
\email{s.barthel@vu.nl}
\author[Yuka Kotorii]{Yuka Kotorii} 
\address[Kotorii]{
Mathematics Program, Graduate School of Advanced Science and Engineering, Hiroshima University
\\
1-7-1 Kagamiyama Higashi-hiroshima City, Hiroshima, Japan 739-8521}
\address[Kotorii]{International Institute for Sustainability with Knotted Chiral Meta Matter (WPI-SKCM$^2$), Hiroshima University, 1-3-1 Kagamiyama, Higashi-Hiroshima, Hiroshima 739-8526, Japan}
\address[Kotorii]{interdisciplinary Theoretical \& Mathematical Sciences Program (iTHEMS) RIKEN
\\
 2-1, Hirosawa, Wako, Saitama 351-0198, Japan}
\email{kotorii@hiroshima-u.jp}
\begin{abstract} 
This paper studies the linking numbers of random links within the grid model. The linking number is treated as a random variable on the isotopy classes of 2-component links, with the paper exploring its asymptotic growth as the diagram size increases. The main result is that the
$u$th moment of the linking number for a random link is a polynomial in the grid size with degree $d\leq u$, and all odd moments vanishing. The limits of the moments of the normalized linking number are computed, and it is shown that the distribution of the normalized linking number converges weakly as the grid size tends to infinity.
\end{abstract} 
\thanks{The author is partially supported by Grant-in-Aid for Young Scientists (B) (No. 16K17586), Grant-in-Aid for Early-Career Scientists (No. 20K14322) and Grant-in-Aid for Scientific Research (C) (No. 25K07005), Japan Society for the Promotion of Science. This work was in part supported by RIKEN iTHEMS Program and International Institute for Sustainability with Knotted Chiral Meta Matter (WPI-SKCM$^2$), Hiroshima University, 1-3-2 Kagamiyama Higashi-hiroshima City, Hiroshima 739-8521 Japan}


\maketitle

\section{Introduction}

In this paper, we study linking numbers of random 2-component oriented links in the grid model, based on the representation of links as grid diagrams as introduced by Brunn \cite{B} and Cromwell \cite{C}.
The grid model has two benefits: First, it is universal, which means that every link can be represented by a grid diagram, and second, it allows for a simple combinatorial description of links in terms of two permutations.


We view the linking number as a random variable on the set of isotopy classes of 2-component links with grid diagrams of grid size $2n \times 2n$ and ask for its asymptotic growth with increasing diagram size as $n$ tends to infinity.
 
The $u$th moment of a random variable $X$ is the expected value $E[X^u]$. The moments of an invariant give a concrete indication of its value on a randomly sampled link. In order to understand the distribution of an invariant in the limit $n \rightarrow \infty$, the invariant must be normalized with respect to $n$. The following theorems determine the order of growth of the linking number.

\begin{theorem}\label{momentthm}
The $u$th moment of the linking number $lk$ for a random link $L_{n}$ in the grid model is a polynomial in $n$ of degree~$d\leq u$. Furthermore, all odd moments equal zero.
\end{theorem}

We also determine the leading terms of the polynomials $E[lk(L_{n})^u]$, which give the limits of the moments of the normalized invariant $lk/n$. 
For $u=2$, we obtain $\lim_{n \rightarrow \infty} E[(lk/n)^{u}] = 1/36$.

\begin{corollary}\label{distribution}
The distribution of the normalized linking number $\frac{lk(L_{n})}{n}$ converges weakly to the unique distribution with moments $\lim_{n \rightarrow \infty} E[(lk/n)^{u}]$ as $n \rightarrow \infty$. That is, the normalized linking number $\frac{lk(L_{n})}{n}$ converges weakly.
\end{corollary}

\begin{remark}
The normalized linking number $\frac{lk}{n}$ also converges weakly in the petaluma model as has been shown by Even, Hass, Linial and Nowik \cite{EHLN}. They conjectured that the normalized invariant $\frac{v_m}{n^m}$ of any finite-type knot invariant $v_m$ of order $m$ converges weakly to a limit distribution in the petaluma model as $n \rightarrow \infty$. This would generalize the result for the normalized linking number since $\frac{lk}{n}$ is a finite-type link invariant of order~$1$. 
\end{remark}

In section 2, we recall the definitions of random knots and links as well as grid diagrams, and define 2-component random grid link diagrams.
In section 3, we recall the linking number of the grid diagram.
In section 4 and 5, we prove Theorem~\ref{momentthm} and Corollary~\ref{distribution} respectively.

\section*{Acknowledgements} 
The authors thank Professors Shin-ichi Ohta and Kenkichi Tsunoda for helpful discussions and their many comments and suggestions.
Y.K. is partially supported by the JSPS KAKENHI Grant-in-Aid for Early-Career Scientists Grant Number 20K14322 and ACT-X.
This work was in part supported by the RIKEN iTHEMS Program and World Premier International Research Center Initiative  program, International Institute for Sustainability with Knotted Chiral Meta Matter (WPI-SKCM$^2$).

\section{Random grid diagrams}

A {\it grid knot diagram of order $n$} is a polygonal knot diagram consisting of $2n$ vertices, $n$ vertical and $n$ horizontal edges in the $n \times n$ grid, where vertical edges always cross over horizontal edges, and each vertical and horizontal line contains exactly one edge of the knot diagram \cite{B,C}. 
Consequently, the knot diagram is determined by the grid coordinates of the vertices of the knot diagram. This information can be given as a pair of permutations $\sigma, \pi \in S_n$ that encodes the $(x, y)$~coordinates of the knot vertices in the grid and the order in which the knot passes through them.
Explicitly, the knot diagram is obtained by connecting the $2n$ vertices given by their coordinates as follows: $(\sigma(1), \pi(1)) \rightarrow (\sigma(1), \pi(2)) \rightarrow (\sigma(2), \pi(2)) \rightarrow (\sigma(2), \pi(3)) \rightarrow \dots \rightarrow (\sigma(n), \pi(1)) \rightarrow (\sigma(1), \pi(1)).$
Note that applying a cyclic permutation to both $\sigma$ and $\pi$ defines the same knot diagram.

In this paper, we study {\it 2-component oriented grid link diagrams $L_{n}$ of order $2n$} consisting of $n$ vertical and $n$ horizontal edges for each component. Such a diagram $L^{\sigma,\pi}_{n}$ can be defined similarly to a grid knot diagram by a pair of permutations $\sigma$, $\pi \in S_{2n}$:
The diagram of the first link component is described by the path $(\sigma(1), \pi(1)) \rightarrow (\sigma(1), \pi(2)) \rightarrow (\sigma(2), \pi(2)) \rightarrow (\sigma(2), \pi(3)) \rightarrow \dots \rightarrow (\sigma(n), \pi(n)) \rightarrow (\sigma(n), \pi(1)) \rightarrow (\sigma(1), \pi(1))$, and the diagram of the second link component is given by
$(\sigma(n+1), \pi(n+1)) \rightarrow (\sigma(n+1), \pi(n+2)) \rightarrow (\sigma(n+2), \pi(n+2)) \rightarrow (\sigma(n+2), \pi(n+3)) \rightarrow \dots \rightarrow (\sigma(2n), \pi(2n)) \rightarrow (\sigma(2n), \pi(n+1)) \rightarrow (\sigma(n+1), \pi(n+1))$. 
The orientation of each component is naturally induced by the order in which the vertices are traversed in the construction. 
For example, the 2-component grid link of order $4$ determined by $\sigma=(1,3,2,4)$ and $\pi=(2,4,1,3)$ is illustrated in Figure~\ref{gridgiagram}.
Its first component is given by $(1,2) \rightarrow (1,4) \rightarrow (3,4) \rightarrow (3,2) \rightarrow (1,2)$ and its second component by $(2,1) \rightarrow (2,3) \rightarrow (4,3) \rightarrow (4,1) \rightarrow (2,1)$.

Cromwell showed that every link (including knots) can be represented in a grid diagram \cite{C2}. 
Adams~et. al. \cite{A} noted that a petal diagram determined by a permutation $\pi \in S_{2n+1}$ can be turned into a grid diagram. The grid diagram $(\sigma, \pi)$ of order $2n+1$ has the same $\pi$ as the petal diagram, and $\sigma$ is defined by $\sigma (m)=nm \pmod{2n+1}$.
A similar grid model that also describes links by permutations but allows for varying number of components was introduced by Farhi~{\it et. al.} \cite{FGHLS}. 

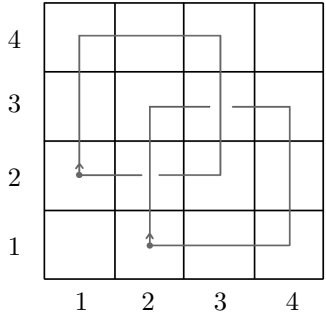
\begin{figure}[h]
  \begin{center}
\includegraphics[width=.3\linewidth]{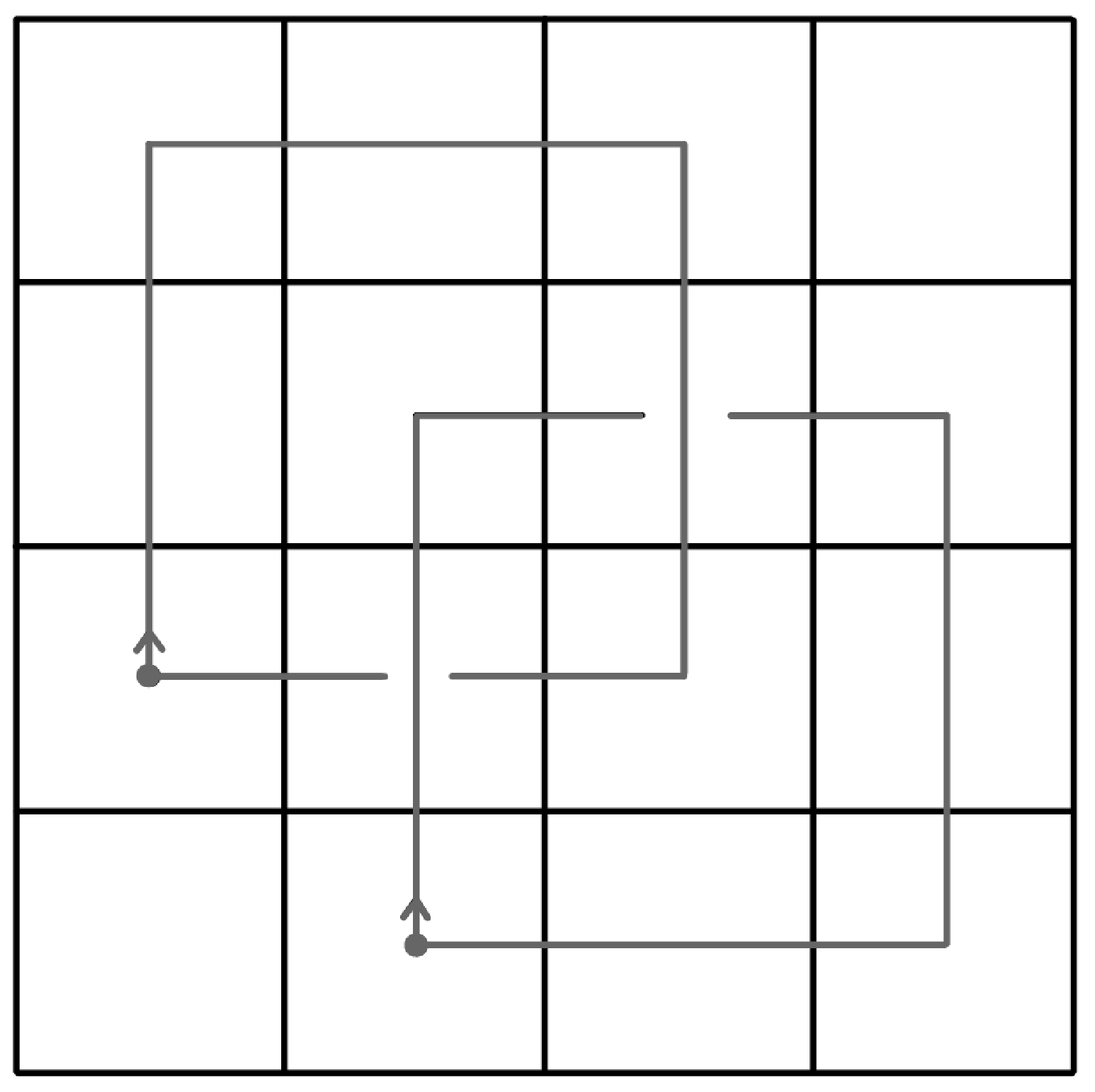}
\put(-95,-10){$1$} \put(-70,-10){$2$} \put(-43,-10){$3$} \put(-16,-10){$4$} 
\put(-120,11){$1$} \put(-120,37){$2$} \put(-120,65){$3$} \put(-120,89){$4$} 
   \caption{The 2-component oriented grid link diagram $L_{2}^{(1,3,2,4),(2,4,1,3)}$ of order 4. The link has linking number +1.}
    \label{gridgiagram}
  \end{center}
\end{figure}

A {\it random knot} in the grid model is obtained by choosing $\sigma$ and $\pi$ independently and uniformly at random~\cite{EHLN}.
Similarly, a {\it 2-component oriented random link $L_n$} in the grid model is defined by choosing $\sigma$ and $\pi$ independently and uniformly at random.

For 2-component random links in the the petaluma model, the limiting distribution of the linking number is determined by Even, Hass, Linial and Nowik \cite{EHLN}.

\section{Linking numbers of grid links} 
The linking number of an oriented 2-component link diagram is defined by the sum of the crossing signs such that the first link component crosses under the second.
The linking numbers of the 2-component grid link diagrams can therefore be computed as described in the following proposition.

\begin{proposition}
Let $\sigma=(x_1, \cdots, x_{n}, x'_1, \cdots, x'_{n})$ and $\pi=(y_1, \cdots, y_{n}, y'_1, \cdots, y'_{n})$ be permutations in $S_{2n}$. Then, the linking number of the 2-component link $L_{n}^{\sigma, \pi}$ of order $2n$ defined by $\sigma$ and $\pi$ is given as
$$lk(L_{n}^{\sigma, \pi})=\sum_{k=1}^{n} \sum_{l=1}^{n} \bm{\varepsilon}_{k,l}^{\sigma, \pi},$$
where $\bm{\varepsilon}_{k,l}^{\sigma, \pi}$ is the sign of the crossing between the horizontal edge from $(x_k, y_{k+1})$ to $(x_{k+1}, y_{k+1})$ and the vertical edge from $(x'_l, y'_{l})$ to $(x'_{l}, y'_{l+1})$ or 0 if the edges do not cross. 
This is, $\bm{\varepsilon}_{k,l}^{\sigma, \pi}$ is 1 if $\sigma$ and $\pi$ satisfy that $x_k < x'_l < x_{k+1}$ and $y'_{l+1} < y_{k+1}< y'_l$, or $x_{k+1} < x'_l < x_{k}$ and $y'_{l} < y_{k+1}< y'_{l+1}$. $\bm{\varepsilon}_{k,l}^{\sigma, \pi}$ is $-1$ if $x_k < x'_l < x_{k+1}$ and $y'_{l} < y_{k+1}< y'_{l+1}$, or $x_{k+1} < x'_l < x_{k}$ and $y'_{l+1} < y_{k+1}< y'_l$. Otherwise, $\bm{\varepsilon}_{k,l}^{\sigma, \pi}$ is 0. Here we consider indices $k$ and $l \mod n$.
\end{proposition}

\begin{remark}\label{remark_def_epsilon}
We call the condition $x_k < x'_l < x_{k+1}$ and $x_{k+1} < x'_l < x_{k}$ the condition $A_{k,l}$ and ${A}^{-1}_{k,l}$ respectively, and $y'_{l+1} < y_{k+1}< y'_l$ and $y'_{l} < y_{k+1}< y'_{l+1}$ the condition $B_{k,l}$ and ${B}^{-1}_{k,l}$ respectively. Obviously $A_{k,l}$ and ${A}^{-1}_{k,l}$,  ($B_{k,l}$ and ${B}^{-1}_{k,l}$, respectively) are mutually exclusive. \\
With this notation, we have $\bm{\varepsilon}_{k,l}^{\sigma, \pi}=
\begin{cases} +1 , & \text{for } A_{k,l} \wedge B_{k,l}\\
+1  , & \text{for } A_{k,l}^{-1} \wedge B_{k,l}^{-1}\\
-1 , & \text{for } A_{k,l} \wedge B_{k,l}^{-1}\\
-1 , & \text{for } A_{k,l}^{-1} \wedge B_{k,l}\\
\phantom{+}0, & \text{otherwise}
\end{cases}$.
\end{remark}

\section{Proof of Theorem \ref{momentthm}}
\begin{lemma}\label{odd_moments_vanish}The odd moments of the linking number vanish: $E[lk(L_n)^u]=0$ for $u$ odd. 
\end{lemma}
\begin{proof}
    For any $\sigma=(\sigma_1, \sigma_2, \cdots ,\sigma_{2n}) \in S_{2n}$, we define $\hat{\sigma} \in S_{2n}$ whose $i$th entry is $2n-\sigma_i +1$. By definition, $\hat{\sigma} \neq \sigma$.
For any $\sigma, \pi \in S_{2n}$, the grid link diagram obtained from $\sigma, \pi$ is the mirror image of the grid link diagram obtained from $\hat{\sigma}, \pi$ (by reflection over a vertical center line of the grid).

Because a crossing of the mirror image is opposite in sign to the corresponding crossing of the original one, these two linking numbers are opposite in sign.  
Therefore, the odd moments of the linking number vanish. 
\end{proof}


Therefore, from now on we only consider the case that the moment $u$ is even.

\begin{lemma}\label{lk1}
Let $\sigma, \pi \in S_{2n}$, and $u$ be even with $u<n$. The $u$th moment of the linking number of the link~$L_{n}^{\sigma, \pi}$ whose grid diagram is defined by $\sigma$ and $\pi$ is given by 
\[E_{\sigma, \pi}(lk(L_{n}^{\sigma, \pi})^{u})= \sum_{k_1, \cdots , k_{u}=1}^{n} \sum_{l_1, \cdots , l_{u}=1}^{n} \sum_{\bm{\varepsilon} \in \{\pm1\}^{u}} \left( \prod_{i=1}^{u} \varepsilon_i \right)   \frac{1}{(2n!)^2} \sharp \{ (\sigma, \pi) \in (S_{2n})^2 \mid \bm{\varepsilon_{k,l}}^{\sigma, \pi} = \bm{\varepsilon} \},\]
where $\bm{\varepsilon}=(\varepsilon_1, \cdots, \varepsilon_{u})$ and $ \bm{\varepsilon_{k,l}}^{\sigma, \pi}=(\varepsilon_{k_1,l_1}^{\sigma, \pi}, \cdots, \varepsilon_{k_{u},l_{u}}^{\sigma, \pi})$.
\end{lemma}
\begin{proof}
    The product of $u$ linking numbers of $L_{n}^{\sigma, \pi}$ is computed as
    \begin{align*}
lk(L_{n}^{\sigma, \pi}) ^{u}
&= (\sum_{k_1=1}^{n} \sum_{l_1=1}^{n} \varepsilon_{k_1,l_1}^{\sigma, \pi}) \cdots (\sum_{k_{u}=1}^{n} \sum_{l_{u}=1}^{n} \varepsilon_{k_{u},l_{u}}^{\sigma, \pi}) \\
&= \sum_{k_1, \cdots , k_{u}=1}^{n} \sum_{l_1, \cdots , l_{u}=1}^{n}  \varepsilon_{k_1,l_1}^{\sigma, \pi} \cdots \varepsilon_{k_{u},l_{u}}^{\sigma, \pi}  \\ 
&= \sum_{k_1, \cdots , k_{u}=1}^{n} \sum_{l_1,  \cdots , l_{u}=1}^{n} \sum_{\bm{\varepsilon} \in \{\pm1\}^{u}} \left( \prod_{i=1}^{u} \varepsilon_i \right) \mathcal{I}[\bm{\varepsilon_{k,l}}^{\sigma, \pi} = \bm{\varepsilon} ], 
\end{align*}
where $\mathcal{I}[-] = 1$ if the condition $\bm{\varepsilon_{k,l}}^{\sigma, \pi} = \bm{\varepsilon}$ is true and $\mathcal{I}[-] = 0$ otherwise. Therefore, the moments can be expressed as    
\begin{align*}
& E_{\sigma, \pi}(lk(L_{n}^{\sigma, \pi})^{u}) \nonumber \\
&= \frac{1}{(2n!)^2} \sum_{\sigma, \pi \in S_{2n}} \left( \sum_{k_1, \cdots , k_{u}=1}^{n} \sum_{l_1, \cdots , l_{u}=1}^{n} \sum_{\bm{\varepsilon} \in \{\pm1\}^{u}}  \left( \prod_{i=1}^{u} \varepsilon_i \right) \mathcal{I}[\bm{\varepsilon_{k,l}}^{\sigma, \pi} = \bm{\varepsilon}] \right) \nonumber \\
&= \sum_{k_1, \cdots , k_{u}=1}^{n} \sum_{l_1, \cdots , l_{u}=1}^{n} \sum_{\bm{\varepsilon} \in \{\pm1\}^{u}}  \left( \prod_{i=1}^{u} \varepsilon_i \right) \frac{1}{(2n!)^2} \sum_{\sigma, \pi \in S_{2n}}\mathcal{I}[\bm{\varepsilon_{k,l}}^{\sigma, \pi} = \bm{\varepsilon}]  \nonumber \\
&= \sum_{k_1, \cdots , k_{u}=1}^{n} \sum_{l_1, \cdots , l_{u}=1}^{n} \sum_{\bm{\varepsilon} \in \{\pm1\}^{u}} \left( \prod_{i=1}^{u} \varepsilon_i \right)   \frac{1}{(2n!)^2} \sharp \{ (\sigma, \pi) \in (S_{2n})^2 \mid \bm{\varepsilon_{k,l}}^{\sigma, \pi} = \bm{\varepsilon} \}.  
\end{align*}  
\end{proof}

The following lemma unpacks the condition of $\bm{\varepsilon_{k,l}}^{\sigma, \pi} = \bm{\varepsilon}$ into an equivalent condition formulated in terms of the crossing behavior $A, A^{-1}, B, B^{-1}$ of the oriented linear segments of the link diagram $L_{n}^{\sigma, \pi}$.

\begin{lemma}\label{condition} 
For any $\bm{\varepsilon}=(\varepsilon_1, \cdots, \varepsilon_{u}) \in \{\pm 1\}^u$, the following three logical conditions are equivalent.\\

\begin{itemize}
    \item[(1)] $\displaystyle \bm{\varepsilon_{k,l}}^{\sigma, \pi} = \bm{\varepsilon}, \text{ i.e, } \bigwedge_{i=1}^{u} \left( {\varepsilon}_{k_i,l_i}^{\sigma, \pi} = \varepsilon_i  \right) $
    \item[(2)] $\displaystyle \bigwedge_{i =1}^{u}  (A_{k_i,l_i} \wedge B_{k_i,l_i}^{\varepsilon_i}) \vee (A_{k_i,l_i}^{-1} \wedge B_{k_i,l_i}^{- \varepsilon_i})$
\item[(3)] $\displaystyle \bigvee_{\bm{\eta} \in \{ \pm1 \}^u} \left( \bigwedge_{i=1}^{u} A_{k_i,l_i}^{\eta_i} \wedge   B_{k_i,l_i}^{\eta_i \varepsilon_i} \right)$, 
where $\bm{\eta}=(\eta_1, \cdots, \eta_{u})$.
\end{itemize}
\end{lemma}
Note that the roles of $A_{k_i,l_i}$ and $B_{k_i,l_i}$ in the above lemma can be exchanged. That is, condition~(2) is true if and only if $\bigwedge_{i =1}^{u}  (A_{k_i,l_i}^{\varepsilon_i} \wedge B_{k_i,l_i}) \vee (A_{k_i,l_i}^{- \varepsilon_i} \wedge B_{k_i,l_i}^{-1})  $ is true, and condition~(3) is true if and only if $\bigvee_{\bm{\eta} \in \{ \pm1 \}^u}  \left( \bigwedge_{i=1}^{u} A_{k_i,l_i}^{\eta_i\varepsilon_i} \wedge   B_{k_i,l_i}^{\eta_i} \right)$ is true.
\begin{proof}
It follows from Remark~\ref{remark_def_epsilon} that condition~(1) and condition~(2) are equivalent.
We show the equivalence of condition~(2) and condition~(3) by induction on $u$. For better readability, we denote $A_{k_i,l_i}$ and $B_{k_i,l_i}$ by $A_i$ and $B_i$, respectively. \\
The base case $u=1$ holds since condition~(2) is obviously equivalent to condition~(3). 
Suppose that the statement of (2) being equivalent to (3) holds for $u-1$, that is,
$$\bigwedge_{i =1}^{u-1}  (A_i \wedge B_i^{\varepsilon_i}) \vee (A_i^{-1} \wedge B_i^{- \varepsilon_i}) 
= \bigvee_{\bm{\eta} \in \{ \pm1 \}^{u-1}}   \left( \bigwedge_{i=1}^{u-1} A_i^{\eta_i} \wedge   B_i^{\eta_i \varepsilon_i} \right), \bm{\eta}=(\eta_1, \cdots, \eta_{u-1}).$$
Then we have
\begin{align*}
& \bigwedge_{i =1}^{u}  (A_i \wedge B_i^{\varepsilon_i}) \vee (A_i^{-1} \wedge B_i^{- \varepsilon_i}) \\
 &= \left( \bigwedge_{i =1}^{u-1}  (A_i \wedge B_i^{\varepsilon_i}) \vee (A_i^{-1} \wedge B_i^{- \varepsilon_i}) \right) \wedge \left( (A_u \wedge B_u^{\varepsilon_u}) \vee (A_u^{-1} \wedge B_u^{- \varepsilon_u}) \right) \\
&= \left( \bigvee_{\bm{\eta} \in \{ \pm1 \}^{u-1}}   \left( \bigwedge_{i=1}^{u-1} A_i^{\eta_i} \wedge   B_i^{\eta_i \varepsilon_i} \right) \right) \wedge \left( (A_u \wedge B_u^{\varepsilon_u}) \vee (A_u^{-1} \wedge B_u^{- \varepsilon_u}) \right) \\
&=   \bigvee_{\bm{\eta} \in \{ \pm1 \}^{u-1}} \left( \bigwedge_{i=1}^{u-1} \left( A_i^{\eta_i} \wedge   B_i^{\eta_i \varepsilon_i} \right) \wedge A_u \wedge B_u^{\varepsilon_u} \right) \vee   \left( \bigwedge_{i=1}^{u-1} \left( A_i^{\eta_i} \wedge B_i^{\eta_i \varepsilon_i} \right) \wedge A_u^{-1} \wedge B_u^{- \varepsilon_u}  \right) \\
&= \bigvee_{\bm{\eta} \in \{ \pm1 \}^u}  \bigwedge_{i=1}^{u} \left( A_i^{\eta_i} \wedge  B_i^{\eta_i \varepsilon_i}  \right).
\end{align*}
Therefore, we obtain the equivalence of (2) and (3).  
\end{proof}

Using the previous result, it is possible to count the number of grid links with given $\bm{\varepsilon_{k,l}}^{\sigma, \pi}$ by considering the orientations of the horizontal edges defined by $\sigma$ and the orientations of the vertical edges defined by $\pi$ independently, as shown in the next lemma. 

\begin{lemma}\label{type} 
Let $\bm{k}=(k_1, \cdots , k_{u})$ and $\bm{l}=(l_1, \cdots , l_{u})$ be sequences on $\{1, \cdots , n \}$ and $\bm{\varepsilon}=(\varepsilon_1, \cdots, \varepsilon_u)$ be a sequence on $\{ \pm1 \}$. 
Then, it holds that $$\sharp \{ (\sigma, \pi) \in (S_{2n})^2 \mid \bm{\varepsilon_{k,l}}^{\sigma, \pi} = \bm{\varepsilon} \} = \sum_{\bm{\eta} \in \{ \pm1 \}^u}  \sharp \{ \sigma \in S_{2n} \mid \bigwedge_{i=1}^{u} A_{k_i,l_i}^{\eta_i} \} 
 \sharp \{ \pi \in S_{2n} \mid \bigwedge_{i=1}^{u} B_{k_i,l_i}^{\eta_i \varepsilon_i} \}. $$ 
\end{lemma} 

\begin{proof} For better readability, we denote $A_{k_i,l_i}$ and $B_{k_i,l_i}$ by $A_i$ and $B_i$, respectively. Since
$\bigwedge_{i=1}^{u} \left( A_i^{\eta_i} \wedge  B_i^{\eta_i \varepsilon_i}  \right) = (\bigwedge_{i=1}^{u} A_{i}^{\eta_i}) \wedge (\bigwedge_{i=1}^{u}  B_{i}^{\eta_i \varepsilon_i})$, Lemma~\ref{condition} implies 
\[\sharp \{ (\sigma, \pi) \in (S_{2n})^2 \mid \bm{\varepsilon_{k,l}}^{\sigma, \pi}=\bm{\varepsilon} \} = \sharp \{ (\sigma, \pi) \in (S_{2n})^2 \mid \bigvee_{\bm{\eta} \in \{ \pm1 \}^u} ( \bigwedge_{i=1}^{u} A_{i}^{\eta_i}) \wedge (\bigwedge_{i=1}^{u}  B_{i}^{\eta_i \varepsilon_i}) \}. (\ast) \]
Because $A_{i}$ and $A_{i}^{-1}$ are mutually exclusive, the conditions $( \bigwedge_{i=1}^{u} A_{i}^{\eta_i}) \wedge (\bigwedge_{i=1}^{u}  B_{i}^{\eta_i \varepsilon_i})$ are mutually exclusive for different $(\eta_1, \cdots ,\eta_u) \in \{\pm1 \}^u$ . Therefore, the above equation can be continued as  
\[(\ast)=\sum_{\bm{\eta} \in \{ \pm1 \}^u}  \sharp \{ (\sigma, \pi) \in (S_{2n})^2 \mid ( \bigwedge_{i=1}^{u} A_{i}^{\eta_i}) \wedge (\bigwedge_{i=1}^{u}  B_{i}^{\eta_i \varepsilon_i}) \},\]
and since $A_{i}$ and $B_{i}$ are independent, this equals
\[  \sum_{\bm{\eta} \in \{ \pm1 \}^u}  \sharp \{ \sigma \in S_{2n} \mid \bigwedge_{i=1}^{u} A_{i}^{\eta_i} \} 
 \sharp \{ \pi \in S_{2n} \mid \bigwedge_{i=1}^{u} B_{i}^{\eta_i \varepsilon_i} \}. \] 
\end{proof}
To calculate  $\sharp \{ \sigma \in S_{2n} \mid \bigwedge_{i=1}^{u} A_{i}^{\eta_i} \}$ and $\sharp \{ \pi \in S_{2n} \mid \bigwedge_{i=1}^{u} B_{i}^{\eta_i \varepsilon_i} \}$ for given $\eta$ and $\varepsilon$, we define the {\it type} of a length $u$ sequence on $\{1, \cdots ,n\}$ as follows.

\begin{definition}
Let $u < n$. For a given sequence $(k_1, \cdots, k_u)$ of length $u$ with $k_1, \cdots, k_u \in \{1, \cdots ,n\}$, we define its type $P$ as the set of sequences of the non-empty subsets whose entries are $\{1, \cdots ,u\}$ by the following rules: \\
(0) Each $i$ ($1 \leq i \leq u$) appears in exactly once in $P$.\\
(1) Indices $i$ and $j$ are in the same set in a sequence in $P$, if $k_i=k_j$. \\
(2) In a sequence in $P$, the next set of the set including $i$ includes $j$, if $k_i+1 \equiv k_j \mod{n}$. \\
(3) Indices $i$ and $j$ are in different sequences, if $k_i-k_j \not\equiv 0, \pm 1 \mod{n}$. \\
We denote the set of all types sequences of length $u$ on $\{1, \cdots ,n\}$ by $T(u, n)$. 
\end{definition}

\begin{example}
The sequence $(2,5,5,1)$ on $\{1, \cdots ,5\}$ is of type $\{ (\{2,3\}, \{4\}, \{1\}) \}$. 
The sequence $(2,4,4,1)$ on $\{1, \cdots ,5\}$ is of type $\{ (\{2,3\}), (\{4\}, \{1\}) \}$. 
The sequence $(4,2,1,4)$ on $\{1, \cdots ,5\}$ is of type $\{ (\{1,4\}), (\{3\}, \{2\}) \}$.
\end{example}

\begin{remark}\label{structure} 
Let $u < n$.  Suppose a sequence $(k_1, \cdots, k_u)$ on $\{1, \cdots ,n\}$ is of type  $P=\{p_1,p_2, \cdots ,p_s \}$, where each $p_h=(P^h_m)_{m=1, \cdots t_h}$ $(1 \leq h \leq s)$ is a sequence and $P^h_m=\{p^h_{m,o} \mid 1 \leq o \leq v_{h,m}  \}$ is a subset of $\{1, \cdots ,u\}$. Here, $s, t_h, v_{h,m}$ are some positive integers.
That is,
\begin{align*}
P=\{p_1, \cdots, &p_h, \cdots ,p_s \} \\
&p_h =(P_1^h, \cdots , P_m^h, \cdots , P^h_{t_h}) \\
&\phantom{p_h =(P_1^h, \cdots ,} P_m^h=\{p_{m,1}^h, \cdots, p^h_{m,o}, \cdots, p_{m,v_{h,m}}^h \}. 
\end{align*}
Note that $p^h_{m,o} \neq p^{h'}_{m',o'}$ if $(h,m,o)\neq (h',m',o')$ and moreover $1, \cdots , u$ appear in $P$ exactly once, that is, $\{p^h_{m,o} \mid h,m,o \} = \{1, \cdots ,u\}$. 
Then, \\
(1) For any $h$ $(1 \leq h \leq s)$ and $m$ $(1 \leq m \leq t_h)$, $$ k_{p^h_{m,1}} =  \cdots = k_{p^h_{m,v_{h,m}}}.$$    
(2) For any $h$ $(1 \leq h \leq s)$ and $m$ $(1 \leq m \leq t_h)$, 
$$k_{p^h_{m,1}} = k_{p^h_{1,1}} + m - 1 \mod{n}.$$
Therefore $k_{p^h_{m,o}} = k_{p^h_{1,1}} +m -1$ for any $h$ $(1 \leq h \leq s)$, $m$ ($1 \leq m \leq t_h$) and $o$ ($1 \leq o \leq v_{h,m}$).\\
(3) For any $h, h'$ ($1 \leq h, h' \leq s$ and $h \neq h'$),  
the distance of $k_{p^{h}_{1,1}}$ and $k_{p^{h'}_{1,1}} +t_{h'} -1$ in $\mod{n}$
is greater or equal to 2. That is,
$$ k_{p^{h'}_{1,1}} +t_{h'} -1 \notin \{k_{p^{h}_{1,1}}-1, \ k_{p^{h}_{1,1}}, \ k_{p^{h}_{1,1}}+1 \} \mod{n}.$$ 
Moreover, the distance of $k_{p^{h'}_{1,1}}$ and $k_{p^{h}_{1,1}} +t_{h} -1$ in $\mod{n}$
is greater or equal to 2. 
\end{remark} 


The following lemma shows that the cardinality of the set $\{ (\sigma, \pi) \in (S_{2n})^2 \mid \bm{\varepsilon_{k,l}}^{\sigma, \pi} = \bm{\varepsilon} \} $ depends only on the types of the sequences $(k_1, \cdots , k_{u} )$ and $(l_1, \cdots , l_{u} )$ on $\{1, \cdots ,n\}$. 
This justifies the notation $ \sharp \{ (\sigma, \pi) \in (S_{2n})^2 \mid \varepsilon_{P,Q}^{\sigma, \pi } = \bm{\varepsilon} \}$ where $P$ and $Q$ are the types of $\bm{k}$ and $\bm{l}$, respectively. 
\begin{lemma}\label{type1} 
Let $\bm{k}=(k_1, \cdots , k_{u})$ and $\bm{k'}=(k'_1, \cdots , k'_{u})$ be both of type $P=(p_1, \cdots , p_s)$. 
Then for any $\bm{l}=(l_1, \cdots , l_{u}) \in \{1, \cdots, n\}^u$, $(\eta_1, \cdots , \eta_u) \in \{\pm 1\}^u$, and $\bm{\varepsilon}=(\varepsilon_1, \cdots, \varepsilon_{u})$, it holds that 
$$\sharp \{ \sigma \in S_{2n} \mid \bigwedge_{i=1}^{u} A_{k_i,l_i}^{\eta_i} \}  = \sharp \{ \sigma \in S_{2n} \mid \bigwedge_{i=1}^{u} A_{k'_i,l_i}^{\eta_i} \} $$
and
$$ \sharp \{ (\sigma, \pi) \in (S_{2n})^2 \mid \bm{\varepsilon_{k,l}}^{\sigma, \pi} = \bm{\varepsilon} \} =  \sharp \{ (\sigma, \pi) \in (S_{2n})^2 \mid \bm{\varepsilon_{k',l}}^{\sigma, \pi} = \bm{\varepsilon} \}. $$
The equations also hold if $\bm{k}$ and $\bm{l}$ switch roles. 
\end{lemma}  

\begin{proof}
\begin{align*}
\bigwedge_{i=1}^{u} A_{k_i,l_i} \Longleftrightarrow \bigwedge_{h=1}^{s} ( \bigwedge_{\alpha \in p_h}  A_\alpha) 
& \Longleftrightarrow \bigwedge_{h=1}^{s} ( \bigwedge_{\alpha \in p_h} x_{k_\alpha} < x'_{l_\alpha} <  x_{k_{\alpha}+1} )  
\end{align*}
where $\alpha \in p_h$ runs through all elements appearing in the sequence $p_h$, see Remark~\ref{structure}. 
Therefore, if $(k_1, \cdots , k_{u})$ and $(k'_1, \cdots , k'_{u})$ are of the same type $P$, one can check that 
$$\sharp \{ \sigma \in S_{2n} \mid \bigwedge_{h=1}^{s} ( \bigwedge_{\alpha \in p_h} x_{k_\alpha} < x'_{l_\alpha} <  x_{k_{\alpha}+1} ) \} = \sharp \{ \sigma \in S_{2n} \mid \bigwedge_{h=1}^{s} ( \bigwedge_{\alpha \in p_{h}} x_{k'_\alpha} < x'_{l_\alpha} <  x_{k'_{\alpha}+1} ) \}. $$
Consequently, $\sharp \{ \sigma \in S_{2n} \mid \bigwedge_{i=1}^{u} A_{k_i,l_i} \} $ depends only on the type. The result for allowing negative exponents is obtained similarly, and the first equation is shown.\\
Similarly, we have that $$\sharp \{ \pi \in S_{2n} \mid \bigwedge_{i=1}^{u} B_{k_i,l_i}^{\eta_i \varepsilon_i} \} =\sharp \{ \pi \in S_{2n} \mid \bigwedge_{i=1}^{u} B_{k'_i,l_i}^{\eta_i \varepsilon_i} \}.$$
Hence the second equation  follows from Lemma~\ref{type}.
\end{proof}

\begin{example}
Let $P=(p_1,p_2)=\{ (\{1\}), (\{2,3\}, \{4\}) \}$ and $(k_1, k_2, k_3, k_4)$ be a sequence of type $P$.
Then, 
\begin{align}
\bigwedge_{i=1}^{4} A_{k_{i},l_{i}} 
& \Longleftrightarrow \bigwedge_{h=1}^{2}  (\bigwedge_{\alpha \in p_h} x_{k_\alpha} < x'_{l_\alpha} <  x_{k_{\alpha}+1} ) \nonumber \\
&  \Longleftrightarrow \bigwedge_{\alpha \in \{1\}} (x_{k_\alpha} < x'_{l_\alpha} <  x_{k_{\alpha}+1} ) \wedge \bigwedge_{\alpha \in \{2,3,4\}} (x_{k_\alpha} < x'_{l_\alpha} <  x_{k_{\alpha}+1} ) \nonumber \\
& \Longleftrightarrow (x_{k_1} < x'_{l_1} <  x_{k_{1}+1}) \wedge (x_{k_2} < x'_{l_2}, \ x'_{l_3} <  x_{k_{2}+1} < x'_{l_4}  <  x_{k_{2}+2}), \label{conditionequ}
\end{align}
where ${k_1}, {k_{1}+1}, {k_2}, {k_{2}+1}, {k_{2}+2}$ are different from one another because of the form of type~$P$.
Therefore, $x_{k_1}, x_{k_{1}+1}, x_{k_2}, x_{k_{2}+1}, x_{k_{2}+2}$ are also different from one another.
For example, consider $(k_1, k_2, k_3, k_4)=(1,3,3,4)$ on $\{1,\cdots ,5\}$, which is of type~$P$.
Then, condition (\ref{conditionequ}) becomes 
\begin{align*}
(x_{1} < x'_{l_1} <  x_{2}) \wedge (x_{3} < x'_{l_2}, \ x'_{l_3} <  x_{4} < x'_{l_4}  <  x_{5}).
\end{align*}
On the other hand, for $(k'_1, k'_2, k'_3, k'_4)=(5,2,2,3)$ on $\{1,\cdots ,5\}$, which is also of type~$P$, condition (\ref{conditionequ}) becomes 
\begin{align*}
(x_{5} < x'_{l_1} <  x_{1}) \wedge (x_{2} < x'_{l_2}, \ x'_{l_3} <  x_{3} < x'_{l_4}  <  x_{4}).
\end{align*}
Therefore, we have  $\sharp \{ \sigma \in S_{10} \mid \bigwedge_{i=1}^{4} A_{k_{i},l_{i}} \} = \sharp \{ \sigma \in S_{10} \mid \bigwedge_{i=1}^{4} A_{k'_{i},l_{i}} \}$ .
\end{example}

Next, we compute the number of sequences of a given type. For this, we first introduce the following notation.

\begin{definition}
Given a type $P \in T(u,n)$, we denote by $S_{n,P}$ the set of all length $u$ sequences on $\{1, \cdots ,n\}$ of type $P$.
\end{definition}

\begin{example}
Let $P=\{ (\{1\}), (\{2,3\}, \{4\}) \}$. Then $S_{5,P}=\{ (1,3,3,4),$ $(2,4,4,5),$ $(3,5,5,1), (4,1,1,2), (5,2,2,3) \}$ and $(1,4,4,5) \in S_{6,P}$.
\end{example}

\begin{lemma}\label{numberoftype} 
Let $P=\{p_1, \cdots , p_s\} \in T(u,n)$.
Then the cardinality of ${S_{n,P}}$ is $$ \frac{n (n-1- \sum_{h=1}^{s} l(p_h))!}{(n-s- \sum_{h=1}^{s} l(p_h))!},$$
where $l(p_h)$ is the length of the sequence $p_h$ ($1 \leq h \leq s$). 
 \end{lemma} 

\begin{proof}
Let $(k_1, \cdots , k_u) \in S_{n,P}$ and $\alpha_h \in P^{h}_{1}$, where $P^{h}_{1}$ is the first element of the sequence $p_h$ (see Remark~\ref{structure} for notation).
Suppose that $k_{\alpha_1} =1$. 
Then there are $\binom{n-1- \sum_{h=1}^{s} l(p_h)}{s-1}(s-1)! $ possibilities to choose $k_{\alpha_h}$ ($h=2, \cdots s$).
Therefore, 
$$ \sharp S_{n,P}= n  \binom{n-1- \sum_{h=1}^{s} l(p_h)}{s-1}(s-1)! =  \frac{n(n-1- \sum_{h=1}^{s} l(p_h))!}{(n-s- \sum_{h=1}^{s} l(p_h))!}. $$
\end{proof}

The following lemma states that the number of grid links that differ by changing the sign of a crossing that corresponds to a length~1 sequence of a type equal.
\begin{lemma}\label{type2} 
Let $\bm{\varepsilon}=(\varepsilon_1, \cdots, \varepsilon_u)$ be a sequence on $\{ \pm1 \}$ and $\bm{l}=(l_1, \cdots , l_{u})$ be a sequence on $\{1, \cdots , n \}$.
For any $\bm{k}=(k_1, \cdots , k_{u})$ whose type contains a length~1 sequence $(\{ i \})$, we have 
\begin{align*}
\sharp \{ (\sigma, \pi) \in (S_{2n})^2 \mid \bm{\varepsilon_{k,l}}^{\sigma, \pi} = \bm{\varepsilon} \} = \sharp \{ (\sigma, \pi) \in (S_{2n})^2 \mid \bm{\varepsilon_{k,l}}^{\sigma, \pi} = \bm{\varepsilon'} \}, 
\end{align*}
where $\bm{\varepsilon'}= (\varepsilon_1, \cdots, \varepsilon_{i-1}, -\varepsilon_i, \varepsilon_{i+1}, \cdots, \varepsilon_u)$.
The equation also holds if $\bm{k}$ and $\bm{l}$ switch roles. 
\end{lemma} 

\begin{proof}
Suppose that a type $P$ contains a sequence $(\{1\})$ and $\bm{\varepsilon}=(1, 1, \cdots ,1)$, $\bm{\varepsilon'}=(-1, 1, \cdots ,1)$ without loss of generality. 
Then, $x_{k_1}$ and $x_{k_1 +1}$ do not appear in $A_{k_i,l_i}^{\pm 1}$ and $B_{k_i,l_i}^{\pm 1}$ ($i=2, \cdots ,u$) .
Therefore, for any $\sigma$ and $\sigma'$ that are related by exchanging the $k_1$th element with the $(k_1+1)$th (mod $n$) element,  a pair $(\sigma, \pi)$ satisfies the condition
$$\bigwedge_{i =2}^{u}  (A_{k_i,l_i} \wedge B_{k_i,l_i}^{\varepsilon_i}) \vee (A_{k_i,l_i}^{-1} \wedge B_{k_i,l_i}^{- \varepsilon_i}) $$
 if and only if $(\sigma', \pi)$ satisfies the same condition. 
On the other hand, ($\sigma, \pi$) satisfies the condition $A_{k_1,l_1}^{\pm}$ if and only if ($\sigma', \pi$) satisfies the condition $A_{k_1,l_1}^{\mp}$,  for any $\pi$.
Therefore, $(\sigma, \pi)$ satisfies $(A_{k_1,l_1} \wedge B_{k_1,l_1}^{\varepsilon_1}) \vee (A_{k_1,l_1}^{-1} \wedge B_{k_1,l_1}^{- \varepsilon_1})$ if and only if $(\sigma', \pi)$ satisfies $(A_{k_1,l_1}^{-1} \wedge B_{k_1,l_1}^{\varepsilon_1})  \vee (A_{k_1,l_1} \wedge B_{k_1,l_1}^{- \varepsilon_1})$.
Together with Lemma~\ref{condition}, it follows that $(\sigma, \pi)$ satisfies the condition $\bm{\varepsilon_{k,l}}^{\sigma, \pi} = \bm{\varepsilon}$ if and only if $(\sigma', \pi)$ satisfies $\bm{\varepsilon_{k,l}}^{\sigma', \pi} = \bm{\varepsilon'}$.
\end{proof}

The results of this section can now be combined to prove Theorem~\ref{momentthm}. In the proof, we will use the notation introduced in the following two definitions.

\begin{definition}\label{Def_sets_arrangements}
   For given $P=\{p_1,\cdots ,p_s\}$, $Q=\{q_1, \cdots ,q_t\}$ and $\bm{\varepsilon}$, we denote by $X_{P,Q}$ the set of all elements of $\sigma=(x_1,\cdots,x_n,x'_1,\cdots,x'_n)$ that appear in condition $\bm{\varepsilon}_{P,Q}^{\sigma, \pi}=\bm{\varepsilon}$. 
 Analogously, we denote the set of all elements of $\pi=(y_1,\cdots,y_n,y'_1,\cdots,y'_n)$ that appear in the condition $\bm{\varepsilon}_{P,Q}^{\sigma, \pi}=\bm{\varepsilon}$ by $Y_{P,Q}$.
\end{definition}

Note that $X_{P,Q}$ and $Y_{P,Q}$ are independent of the choice of $\bm{\varepsilon}$, and that $u < n$ by assumption. 
The cardinality of $X_{P,Q}$ is therefore $ \sharp X_{P,Q}= \sum_{h=1}^{s}(l(p_h)+1) + \sum_{h=1}^{t} l(q_h)$, and the cardinality of $Y_{P,Q}$ is $\sharp Y_{P,Q}= \sum_{h=1}^{s} l(p_h) + \sum_{h=1}^{t}(l(q_h)+1)$.

\begin{example}\label{ex_X_PQ}
   For $P=\{(\{1,2\})\}$ and $Q=\{(\{1\}, \{2\})\}$, we have 
\begin{align*}
 X_{P,Q}=\{x_{k_1}, x_{k_1 +1}, x'_{l_1}, x'_{l_1 +1} \}, \ Y_{P,Q}=\{y_{k_1+1}, y'_{l_1},y'_{l_1+1}, y'_{l_1+2}\}.
\end{align*}

\end{example}

\begin{definition}\label{Def_cardinality_arrangements}
Let $N_{P,Q,\bm{\varepsilon}}$ be the total number of arrangements of elements of $ X_{P,Q} \cup Y_{P,Q}$ such that the condition $\bm{\varepsilon}_{P,Q}^{\sigma, \pi} = \bm{\varepsilon}$ is satisfied.
\end{definition}

\begin{remark}\label{Rem_product}
The proof of Lemma~\ref{type} provides a framework to express \(N_{P,Q,\bm{\varepsilon}}\) in terms of the number of ways the elements of $X_{P,Q}$ and $Y_{P,Q}$ can be ordered to satisfy the condition $\bigwedge_{i=1}^{u} A_{i}^{\delta_i}$, or $\bigwedge_{i=1}^{u} B_{i}^{\delta_i}$, respectively, for a given $\bm{\delta} =(\delta_1, \dots, \delta_u) \in \{\pm 1\}^u$. Denoting these sets of possible orders as $X_{P,Q,\bm{\delta}}$ and $Y_{P,Q,\bm{\delta}}$, respectively, we obtain:
\[N_{P,Q,\bm{\varepsilon}} = \sum_{\bm{\eta} \in \{ \pm 1 \}^u}  \sharp X_{P,Q,\bm{\eta}} \ \sharp Y_{P,Q,\bm{\eta}\bm{\varepsilon}} .\]
\end{remark}

\begin{proof}[Proof of Theorem~\ref{momentthm}]
By lemma~\ref{odd_moments_vanish}, the odd moments of the linking number vanish. Therefore, only the even $u$th moments with $u<n$ are considered in the following. \\
According to 
Lemma~\ref{lk1}, for $\sigma, \pi \in S_{2n}$, the $u$th moment of the linking number of the link~$L_{n}^{\sigma, \pi}$ whose grid diagram is defined by $\sigma$ and $\pi$ is given as\\
\begin{equation}\label{E1}
\begin{split}
    &\hspace{13 pt} E_{\sigma, \pi}(lk(L_{n}^{\sigma, \pi})^{u})\\&= \sum_{k_1, \cdots , k_{u}=1}^{n} \sum_{l_1, \cdots , l_{u}=1}^{n} \sum_{\bm{\varepsilon} \in \{\pm1\}^{u}} \left( \prod_{i=1}^{u} \varepsilon_i \right)   \frac{1}{(2n!)^2} \sharp \{ (\sigma, \pi) \in (S_{2n})^2 \mid \bm{\varepsilon_{k,l}}^{\sigma, \pi} = \bm{\varepsilon} \}.
\end{split}
\end{equation}
By Lemma~\ref{condition} and Lemma~\ref{type}, the last term can be rewritten as 
\begin{align*}
    \sharp \{ (\sigma, \pi) \in (S_{2n})^2 \mid \bm{\varepsilon_{k,l}}^{\sigma, \pi} = \bm{\varepsilon} \} &\overset{\ref{condition}}{=} \sharp \{ (\sigma, \pi) \in (S_{2n})^2 \mid \displaystyle \bigvee_{\bm{\eta} \in \{ \pm1 \}^u} ( \bigwedge_{i=1}^{u} A_{k_i,l_i}^{\eta_i} \wedge   B_{k_i,l_i}^{\eta_i \varepsilon_i} )\}\\
    &\overset{\ref{type}}{=} \sum_{\bm{\eta} \in \{ \pm1 \}^u}  \sharp \{ \sigma \in S_{2n} \mid \bigwedge_{i=1}^{u} A_{k_i,l_i}^{\eta_i} \} 
 \sharp \{ \pi \in S_{2n} \mid \bigwedge_{i=1}^{u} B_{k_i,l_i}^{\eta_i \varepsilon_i} \}.
\end{align*}
According to Lemma~\ref{type1}, the factors $\sharp \{ \sigma \in S_{2n} \mid \bigwedge_{i=1}^{u} A_{k_i,l_i}^{\eta_i} \} $ and $\sharp \{ \pi \in S_{2n} \mid \bigwedge_{i=1}^{u} B_{k_i,l_i}^{\eta_i \varepsilon_i}\}$ depend only on the types of $\bm{k}$ and $\bm{l}$.\\
Applying Lemma~\ref{type1} and Lemma~\ref{numberoftype} to equation~(\ref{E1}) allows to write the $u$th moment of the linking number as
\begin{align}
& \hspace{17 pt} E_{\sigma, \pi}(lk(L_{n}^{\sigma, \pi})^{u}) \notag\\ 
&\overset{\ref{type1}}{=} \sum_{P\in T(u,n)}  \sum_{Q\in T(u,n)} \sharp S_{n,P} \sharp S_{n,Q} \sum_{\bm{\varepsilon} \in \{\pm1\}^{u}} \left( \prod_{i=1}^{u} \varepsilon_i \right)  \frac{1}{(2n)!^2} \sharp \{ (\sigma, \pi) \in (S_{2n})^2 \mid \bm{\varepsilon}_{P,Q}^{\sigma, \pi} = \bm{\varepsilon} \} \notag\\
&\overset{\ref{numberoftype}}{=}   \sum_{P\in T(u,n)}  \sum_{Q\in T(u,n)}  \frac{ n^2 (n-1- \sum_{h=1}^{s} l(p_h))! \cdot (n-1- \sum_{h=1}^{t} l(q_h))!}{(n-s- \sum_{h=1}^{s} l(p_h))! \cdot (n-t- \sum_{h=1}^{t} l(q_h))!} \notag \\
& \qquad \sum_{\bm{\varepsilon} \in \{\pm1\}^{u}} \left( \prod_{i=1}^{u} \varepsilon_i \right)  \frac{1}{(2n)!^2} \sharp \{ (\sigma, \pi) \in (S_{2n})^2 \mid \bm{\varepsilon}_{P,Q}^{\sigma, \pi} = \bm{\varepsilon} \} \label{E2}.
\end{align}


Using the notation from Definition~\ref{Def_sets_arrangements} and Definition~\ref{Def_cardinality_arrangements}, the final term of~(\ref{E2}) is expressed as follows, with the first equality implied by Remark~\ref{Rem_product}.

\begin{align*}
 \sharp \{ (\sigma, \pi) \in (S_{2n})^2 \mid \bm{\varepsilon}_{P,Q}^{\sigma, \pi}=\bm{\varepsilon} \} 
&= \binom{2n}{\sharp X_{P,Q}} (2n-\sharp X_{P,Q})! \binom{2n}{\sharp Y_{P,Q}} (2n- \sharp Y_{P,Q})! N_{P,Q,\bm{\varepsilon}} \\
&= \frac{(2n)!^2 N_{P,Q,\bm{\varepsilon}}}{(\sharp X_{P,Q})! (\sharp Y_{P,Q})!}, \\
\end{align*}
which results in the last sum of (\ref{E2}) taking the form
\begin{align}
    &\quad \sum_{\bm{\varepsilon} \in \{\pm1\}^{u}} \left( \prod_{i=1}^{u} \varepsilon_i \right)  \frac{1}{(2n)!^2} \sharp \{ (\sigma, \pi) \in (S_{2n})^2 \mid \bm{\varepsilon}_{P,Q}^{\sigma, \pi} = \bm{\varepsilon}\} \label{sum1}\\
    &= \sum_{\bm{\varepsilon} \in \{ \pm1 \}^{u}} \left( \prod_{i=1}^{u} \varepsilon_i \right)  \frac{N_{P,Q,\bm{\varepsilon}}}{(\sharp X_{P,Q})! (\sharp Y_{P,Q})!} \label{sum2}.
\end{align}


By Lemma \ref{type2}, for a type $P$ that contains a length~1 sequence $(\{ i \})$ as an element, and any $Q$ and $\bm{\varepsilon}=(\varepsilon_1, \cdots, \varepsilon_u)$, it holds that 
$$\left( \prod_{i=1}^{u} \varepsilon_i \right) \sharp \{ (\sigma, \pi) \in (S_{2n})^2 \mid \bm{\varepsilon_{k,l}}^{\sigma, \pi} = \bm{\varepsilon} \} + \left( \prod_{i=1}^{u} \varepsilon'_i \right) \sharp \{ (\sigma, \pi) \in (S_{2n})^2 \mid \bm{\varepsilon_{k,l}}^{\sigma, \pi} = \bm{\varepsilon'} \}=0, $$
where $\bm{\varepsilon'}=(\varepsilon'_1, \cdots, \varepsilon'_u)=(\varepsilon_1, \cdots, \varepsilon_{i-1}, -\varepsilon_i, \varepsilon_{i+1}, \cdots \varepsilon_u)$.
Therefore the sum~(\ref{sum1}) vanishes if $P$ is a type that contains a length ~1 sequence.
With this observation and using the notation of (\ref{sum2}), the computation of the $u$th moment of the linking number continues with
\begin{align}\label{E3}
(\ref{E2}) &= \sum_{\substack{P=(p_1, \cdots ,p_s)  \\ |p_h| \geq 2 \ (\forall h)}}  \sum_{\substack{Q=(q_1, \cdots ,q_t)  \\ |q_h| \geq 2 \ (\forall h) }} n^2 \frac{(n-1- \sum_{h=1}^{s} l(p_h))!}{(n-s- \sum_{h=1}^{s} l(p_h))!}  \frac{(n-1- \sum_{h=1}^{t} l(q_h))!}{(n-t- \sum_{h=1}^{t} l(q_h))!}\\
& \quad \sum_{\bm{\varepsilon} \in \{ \pm1 \}^{u}} \left( \prod_{i=1}^{u} \varepsilon_i \right)  \frac{N_{P,Q,\bm{\varepsilon}}}{(\sharp X_{P,Q})! (\sharp Y_{P,Q})!}, \notag
\end{align}
where $|\cdot|$ denotes the number of elements in a sequence.\\
Since $N_{P,Q,\bm{\varepsilon}}$, $\sharp X_{P,Q}$, and $\sharp Y_{P,Q}$ are independent of the order $n$ of a grid diagram, it follows from (\ref{E3}) that the $u$th moment of the linking number of $L_{n}$ is a polynomial in $n$. Moreover, the degree $d$ of the polynomial is at most $u$, and equality is obtained when $s$ and $t$ take the largest possible values, that is, when $s=t=\frac{u}{2}$. 
\end{proof}

\begin{remark}\label{remark:coefficient}
The last step of the above proof implies the form of the leading coefficient $a_u$ of the polynomial $E_{\sigma, \pi}(lk(L_{n}^{\sigma, \pi})^{u})$. For $u$ odd, the polynomial vanishes. For $u$ even and $s=t=\frac{u}{2}$ it holds that $\sharp X_{P,Q}=\sharp Y_{P,Q}=\frac{u}{2} + \sum_{h=1}^{\frac{u}{2}}( l(p_h)+l(q_h))$, and we denote this number by $n_{P,Q}$. Then the coefficient of $n^{u}$ is given by
\[
a_u =    \sum_{\substack{P=(p_1, \cdots ,p_\frac{u}{2})  \\ |p_h| = 2 \ (\forall h)}}  \sum_{\substack{Q=(q_1, \cdots ,q_\frac{u}{2})  \\ |q_h| = 2 \ (\forall h) }} 
    \sum_{\bm{\varepsilon} \in \{\pm 1\}^{u}} \left( \prod_{i=1}^{u} \varepsilon_i \right) 
    \frac{N_{P,Q,\bm{\varepsilon}}}{(n_{P,Q}!)^2} .
\]
\end{remark}

As an example, we calculate $a_u$ for $u=2$.
\begin{example}
For $u=2$, three types are possible: $\{(\{1,2\})\}$ $\{(\{1\}, \{2\})\}$ and $\{(\{2\}, \{1\})\}$.
We obtain the coefficient by considering the nine possibilities of $P$ and $Q$ to be of either type.
\begin{itemize}
\item[(1)] For $P=Q=\{(\{1,2\})\}$, we have $n_{P,Q}=3$ and 
\begin{align*}
&X_{P,Q,(1,1)}=\{(x_{k_1}< x'_{l_1}< x_{k_1 +1})\}, \\
&X_{P,Q,(-1,1)}=\emptyset, \\
&Y_{P,Q,(1,1)}=\{(y'_{l_1 +1}< y_{k_1 +1}< y'_{l_1}) \}, \\
&Y_{P,Q,(-1,1)}=\emptyset,  
\end{align*}
Therefore
\begin{align*}
\sharp X_{P,Q,(1,1)}=1, \ \sharp X_{P,Q,(-1,1)}=0, \ 
\sharp Y_{P,Q,(1,1)}=1, \ \sharp Y_{P,Q,(-1,1)}=0.
\end{align*} 
Since $\sharp X_{P,Q,(i,j)}= \sharp X_{P,Q,(-i,-j)}$ and $\sharp Y_{P,Q,(i,j)}= \sharp Y_{P,Q,(-i,-j)}$, the above calculation is sufficient to determine $N_{P,Q,(1,1)}$:
\begin{align*}
N_{P,Q,(1,1)}=&\sharp X_{P,Q,(1,1)}\sharp Y_{P,Q,(1,1)}+\sharp X_{P,Q,(-1,1)}\sharp Y_{P,Q,(-1,1)} \\
&+\sharp X_{P,Q,(1,-1)}\sharp Y_{P,Q,(1,-1)}+\sharp X_{P,Q,(-1,-1)}\sharp Y_{P,Q,(-1,-1)} \\
=&2(\sharp X_{P,Q,(1,1)}\sharp Y_{P,Q,(1,1)}+\sharp X_{P,Q,(-1,1)}\sharp Y_{P,Q,(-1,1)}) \\
=&2(1 \cdot 1 + 0  \cdot 0)\\=&2.
\end{align*}
Similarly, 
\begin{align*}
&N_{P,Q,(-1,-1)}=2(1 \cdot 1 + 0  \cdot 0) =2, \\
&N_{P,Q,(-1,1)}=N_{P,Q,(1,-1)}=2(1 \cdot 0 + 0  \cdot 1) =0.
\end{align*}
Therefore we obtain
\begin{align*}
 \sum_{\bm{\varepsilon} \in \{\pm 1\}^{2}} \left( \prod_{i=1}^{2} \varepsilon_i \right)  \frac{N_{P,Q,\bm{\varepsilon}}}{(n_{P,Q}!)^2} = \frac{1}{9}.
\end{align*}
\item[(2)] For $P=\{(\{1,2\})\}$ and $Q=\{(\{1\}, \{2\})\}$, we have $n_{P,Q}=4$ and 

\begin{align*}
& X_{P,Q,(1,1)}=\{(x_{k_1}< x'_{l_1}< x'_{l_1 +1}< x_{k_1 +1}), (x_{k_1}<x'_{l_1 +1}<  x'_{l_1}<x_{k_1 +1}) \}, \\ 
& X_{P,Q,(-1,1)}=\emptyset, \\
& Y_{P,Q,(1,1)}=\emptyset, \\
& Y_{P,Q,(-1,1)}=\{(y'_{l_1}<y'_{l_1+2}<y_{k_1+1}<y'_{l_1+1}),(y'_{l_1+2}<y'_{l_1}<y_{k_1+1}<y'_{l_1+1})\},\\
\end{align*}
Therefore
\begin{align*}
\sharp X_{P,Q,(1,1)}=2, \ \sharp X_{P,Q,(-1,1)}=0, \ 
\sharp Y_{P,Q,(1,1)}=0, \ \sharp Y_{P,Q,(-1,1)}=2.
\end{align*} 
Since $\sharp X_{P,Q,(i,j)}= \sharp X_{P,Q,(-i,-j)}$ and $\sharp Y_{P,Q,(i,j)}= \sharp Y_{P,Q,(-i,-j)}$, the above calculation is sufficient to determine
\begin{align*}
&N_{P,Q,(1,1)}=N_{P,Q,(-1,-1)}=2(2 \cdot 0 + 0  \cdot 2)=0 \\
&N_{P,Q,(-1,1)}=N_{P,Q,(1,-1)}=2(2 \cdot 2 + 0  \cdot 0)=8, \\
\end{align*}

and we obtain
\begin{align*}
 \sum_{\bm{\varepsilon} \in \{\pm 1\}^{2}} \left( \prod_{i=1}^{2} \varepsilon_i \right)  \frac{N_{P,Q,\bm{\varepsilon}}}{(n_{P,Q}!)^2} =  - \frac{1}{36}.
\end{align*}
For $P=\{(\{1\}, \{2\})\}$ and $Q=\{(\{1,2\})\}$, $P=\{(\{1,2\})\}$ and $Q=\{(\{2\}, \{1\})\}$, and $P=\{(\{2\}, \{1\})\}$ and $Q=\{(\{1,2\})\}$, we obtain the same value as for $P=\{(\{1,2\})\}$ and $Q=\{(\{1\}, \{2\})\}$ by a calculation that is very similar to the calculation above.
\item[(3)]For $P=Q=\{(\{1\}, \{2\})\}$, we have $n_{P,Q}=5$ and 
\begin{align*}
X_{P,Q,(1,1)}=&\{(x_{k_1}< x'_{l_1}< x_{k_1+1}<x'_{l_1 +1}< x_{k_1 +2})\}, \\
X_{P,Q,(-1,1)}=&\{(x_{k_1+1}< x'_{l_1}< x_{k_1}<x'_{l_1 +1}< x_{k_1 +2}),(x_{k_1+1}< x'_{l_1}<x'_{l_1 +1}< x_{k_1}< x_{k_1 +2}),\\ &(x_{k_1+1}<x'_{l_1 +1}< x'_{l_1}< x_{k_1}< x_{k_1 +2}),(x_{k_1+1}< x'_{l_1}<x'_{l_1 +1}< x_{k_1 +2}< x_{k_1}), \\
&(x_{k_1+1}<x'_{l_1 +1}< x'_{l_1} < x_{k_1 +2} 
< x_{k_1}),(x_{k_1+1}< x'_{l_1 +1}< x_{k_1 +2}< x'_{l_1}<x_{k_1})\}\\
Y_{P,Q,(1,1)}=&\{(y'_{l_1+2}< y_{k_1+2}< y'_{l_1+1}<y_{k_1 +1}< y'_{l_1})\}, \\
Y_{P,Q,(-1,1)}=&\{(y'_{l_1}< y_{k_1+1}< y'_{l_1+2}<y_{k_1 +2}< y'_{l_1+1}),(y'_{l_1}< y'_{l_1+2}< y_{k_1+1}<y_{k_1 +2}< y'_{l_1+1}), \\
&(y'_{l_1+2}<y'_{l_1}< y_{k_1+1}<y_{k_1 +2}< y'_{l_1+1}),(y'_{l_1}<y'_{l_1+2}<y_{k_1 +2}< y_{k_1+1}< y'_{l_1+1}), \\
&(y'_{l_1+2}<y'_{l_1}<y_{k_1 +2}< y_{k_1+1}< y'_{l_1+1}),
(y'_{l_1+2}<y_{k_1 +2}<y'_{l_1}< y_{k_1+1}< y'_{l_1+1})\}. 
\end{align*}
Therefore 
\begin{align*}
\sharp X_{P,Q,(1,1)}=1, \ \sharp X_{P,Q,(-1,1)}=6, \ 
\sharp Y_{P,Q,(1,1)}=1, \ \sharp Y_{P,Q,(-1,1)}=6.
\end{align*}
Since $\sharp X_{P,Q,(i,j)}= \sharp X_{P,Q,(-i,-j)}$ and $\sharp Y_{P,Q,(i,j)}= \sharp Y_{P,Q,(-i,-j)}$, the above calculation is sufficient to determine
\begin{align*}
&N_{P,Q,(1,1)}=N_{P,Q,(-1,-1)}=2(1 \cdot 1 + 6  \cdot 6) =74, \\
&N_{P,Q,(-1,1)}=N_{P,Q,(1,-1)}=2(1 \cdot 6 + 6  \cdot 1) = 24
\end{align*}
and we obtain
\begin{align*}
 \sum_{\bm{\varepsilon} \in \{\pm 1\}^{2}} \left( \prod_{i=1}^{2} \varepsilon_i \right)  \frac{N_{P,Q,\bm{\varepsilon}}}{(n_{P,Q}!)^2} =   \frac{1}{144}.
\end{align*}
\end{itemize}
For $P=\{(\{1\}, \{2\})\}$ and $Q=\{(\{2\}, \{1\})\}$, $P=\{(\{2\}, \{1\})\}$ and $Q=\{(\{1\}, \{2\})\}$, and $P=Q=\{(\{2\}, \{1\})\}$, we obtain the same value as for $P=Q=\{(\{1\}, \{2\})\}$ by a calculation that is very similar to the calculation above.

Combining the above four cases shows that according to Remark~\ref{remark:coefficient}, for $u=2$ the leading coefficient of the polynomial is
\[
a_2 = 1 \frac{1}{9}  + 4  \left(- \frac{1}{36}\right) + 4 \frac{1}{144}= \frac{1}{36}.
\]

\end{example}

\section{Proof  of corollary \ref{distribution}}

A sequence of distribution functors $F_n$, $n\in \mathbb{N}$ is said to {\it converge weakly} to a limit $F$ if the sequence converges for all continuity points $y$ of $F$, i.e. $\lim_{n \rightarrow \infty} F_n(y)=F(y)$.
A sequence of random variables $X_n$, $n\in \mathbb{N}$ {\it converges weakly} to a limit $X$ if their distribution functions $F_n(x)=P(X_n \leq x)$ converge weakly. 

The following theorem provides a sufficient condition for weak convergence of a sequence of random variables (for example, see \cite{D}).

\begin{theorem}\label{limittheorem}
Suppose $\int x^u dF_n(x)$ has a limit $\mu_u$ for each $u$ and 
$$\limsup_{u \rightarrow \infty} \frac{({\mu_{2u})}^{1/2u}}{2u} < \infty, $$
then $F_n$ converges weakly to the unique distribution with these moments.  
\end{theorem}

\begin{proof}[Proof of Corollary \ref{distribution}]
Let $\mu_u$ be the constant term of $n$ of $u$th moment of $\frac{lk(L_{n})}{n}$.  
(i.e $\mu_u = \lim_{n \rightarrow \infty} E ((\frac{lk(L_{n})}{n})^u$)).
Then we have 
\begin{align*}
\mu_{2u} & = \sum_{\substack{P \in T(2u,n) \\ |p_h| = 2 \ (\forall h)}}  \sum_{\substack{Q \in T(2u,n) \\ |q_h| = 2 \ (\forall h) }}   \sum_{\bm{\varepsilon} \in \{\pm 1\}^{2u}} \left( \prod_{i=1}^{2u} \varepsilon_i \right) 
    \frac{N_{P,Q,\bm{\varepsilon}}}{(n_{P,Q}!)^2}  \\
& \leq \sum_{\substack{P \in T(2u,n) \\ |p_h| = 2 \ (\forall h)}}  \sum_{\substack{Q \in T(2u,n) \\ |q_h| = 2 \ (\forall h) }}  2^{2u}   \frac{N_{P,Q,\bm{\varepsilon}}}{(n_{P,Q}!)^2} \\
& \leq  \left( \frac{(2u)!}{2^u u!} 3^{u} \right)^2 2^{2u}  \\
& = \frac{(2u)!^2 3^{2u} }{u!^2}.
\end{align*}
This implies
\begin{align*}
\limsup_{u \rightarrow \infty} \frac{({\mu_{2u})}^{1/2u}}{2u} 
& \leq \limsup_{u \rightarrow \infty} \left( \frac{(2u)!^2 3^{2u} }{u!^2} \right)^{1/2u} \frac{1}{2u} \\
& = \limsup_{u \rightarrow \infty}  \frac{3 (2u  (2u-1) \cdots (u+1))^{1/u}}{2u} \\
& \leq \limsup_{u \rightarrow \infty}  \frac{3((2u)^u)^{1/u}}{2u} \\
& = 3.
\end{align*} 
By Theorem~\ref{limittheorem}, $P\left(\frac{lk(L_{n})}{n} \leq x\right)$ converges weakly to the unique distribution with moments $\mu_u$, that is, the sequence of the random variables $\frac{lk(L_{n})}{n}$ converges weakly.
\end{proof}




\end{document}